\documentclass[12pt,reqno]{article}

\usepackage[usenames]{color}
\usepackage{amssymb}
\usepackage{amsmath}
\usepackage{amsthm}
\usepackage{amsfonts}
\usepackage{amscd}
\usepackage{graphicx}
\usepackage{lscape}

\usepackage[colorlinks=true,
linkcolor=webgreen,
filecolor=webbrown,
citecolor=webgreen]{hyperref}

\definecolor{webgreen}{rgb}{0,.5,0}
\definecolor{webbrown}{rgb}{.6,0,0}

\usepackage{color}
\usepackage{fullpage}
\usepackage{float}

\usepackage{graphics}
\usepackage{latexsym}
\usepackage{epsf}
\usepackage{breakurl}

\def\modd#1 #2{#1\ \mbox{\rm (mod}\ #2\mbox{\rm )}}
\def\uni{ \, \cup \, }

\def\Enn{\mathbb{N}}

\makeatletter
\def\@fnsymbol#1{\ensuremath{\ifcase#1\or *\or 
   \mathsection\or \mathparagraph\or \|\or **\or \dagger\dagger
   \or \ddagger\ddagger \else\@ctrerr\fi}}
\makeatother

\def\smath#1{\text{\scalebox{.8}{$#1$}}}
\def\sfrac#1#2{\smath{\frac{#1}{#2}}}

\begin{document}

\theoremstyle{plain}
\newtheorem{theorem}{Theorem}
\newtheorem{corollary}[theorem]{Corollary}
\newtheorem{lemma}[theorem]{Lemma}
\newtheorem{proposition}[theorem]{Proposition}

\theoremstyle{definition}
\newtheorem{definition}[theorem]{Definition}
\newtheorem{example}[theorem]{Example}
\newtheorem{conjecture}[theorem]{Conjecture}
\newtheorem{openproblem}[theorem]{Open Problem}

\theoremstyle{remark}
\newtheorem{remark}[theorem]{Remark}

\title{Rarefied Thue-Morse Sums \\
Via Automata Theory and Logic}

\author{
Jeffrey Shallit\footnote{Research funded by a grant from NSERC, 2018-04118.} \\
School of Computer Science\\
University of Waterloo\\
200 University Ave. W. \\
Waterloo, ON  N2L 3G1\\
Canada\\
\href{mailto:shallit@uwaterloo.ca}{\tt shallit@uwaterloo.ca}\\
}

\maketitle

\begin{abstract}
Let $t(n)$ denote the number of $1$-bits in the base-$2$ representation
of $n$, taken modulo $2$.
We show how to prove the classic conjecture of Leo Moser,
on the rarefied sum $\sum_{0 \leq i<n} (-1)^{t(3i)}$,
using tools from automata theory and logic.  The same technique
can be used to prove results about analogous sums.
\end{abstract}

\section{Introduction}

In the 1960's Leo Moser observed that among the first few multiples of
$3$, most of them have an even number of $1$'s when expressed in base $2$.
This is true, for example, for $0,3,6,9,12,15,$ and $18$; the first
counterexample is $21$.   He conjectured that among the first
$n$ multiples of $3$, those with an even number of $1$'s always
predominate over those with an odd number of $1$'s.  This is
{\it Moser's conjecture}.

Let $t(n)$ denote the number of $1$'s appearing in the base-$2$
representation of $n$, taken modulo $2$.
Then $(t(n))_{n \geq 0}$ is
the classic Thue-Morse sequence \cite{Allouche&Shallit:1999}.
Let $b$ be a positive integer,
let $0 \leq j < b$, and for $n \geq 0$ define
$f_{b,j} (n) = \sum_{0 \leq i < n} (-1)^{t(bi+j)} $. 
An expression like $f_{b,j}(n)$ is called a
{\it rarefied Thue-Morse sum} \cite{Goldstein&Kelly&Speer:1992},
and Moser's conjecture can be restated as the claim that
$f_{3,0} (n) > 0$ for all $n \geq 1$.

Moser's conjecture was proved by Newman \cite{Newman:1969}, who
proved even more:  he determined the asymptotic behavior of the
excess.
Newman proved that
$$ \sfrac{1}{20} \leq f_{3,0} (n)/n^e \leq 5,$$
where $e = (\log 3)/(\log 4) \doteq 0.79248$.
Later, the exact behavior of $f_{3,0} (n)$ was determined
by Coquet \cite{Coquet:1983}.

Since Newman's 1969 paper, our knowledge of the behavior of $f_{b,j} (n)$ has
greatly increased.   See, for example,
\cite{Newman&Slater:1975,Coquet:1983,Dumont:1983,
Goldstein&Kelly&Speer:1992,Grabner:1993,Drmota&Skalba:1995,
Grabner&Herendi&Tichy:1997,Tenenbaum:1997,
Dartyge&Tenenbaum:2005,Drmota&Stoll:2008,Shevelev:2008,Shevelev:2009,Hofer:2011,Boreico&El-Baz&Stoll:2014}.

In this paper we show how to obtain bounds like Newman's through an
entirely different approach, via automata theory and logic.
The crucial idea is that, in some cases, there is a deterministic
finite automaton (DFA) that computes $f_{b,j} (n)$ in the
following sense:   it takes representations of
integers $n$ and $y$ in parallel, as inputs, and accepts
iff $y = f_{b,j} (n)$.  Such an automaton is sometimes
called {\it synchronized} \cite{Shallit:2021h}.
The novelty is that the particular
bases of representation for $n$ and $y$ may be different, and
they depend on $b$.  Once we have the synchronized automaton, we can use
existing theorems about synchronization to understand the growth
rate of $f_{b,j} (n)$.   Although typically our method does not
provide the best possible constants, in many cases it provides
decent estimates with very little work.

\section{Notation}

We need some notation to express base-$b$ representation.
Define $\Sigma_b = \{ 0,1,\ldots, b-1 \}$, the usual set of
digits for base $b$.   If $x \in \Sigma_b^*$, say
$x = a_1 a_2 \cdots a_t$, then we define
$[x]_b = \sum_{1 \leq i \leq t} a_i b^{t-i}$.  
By $(n)_b$ for $n \in \Enn$, we mean the word in $\Sigma_b^*$
giving the (unique) most-significant-digit-first 
representation of $n$ in base $b$, with no leading zeros.
Thus, for example, $[00101011]_2 = 43$ and
$(43)_2 = 101011$.

In Section~\ref{sec5} we will also need representations for negative
integers.  To do so, we (analogously) define representation in
base $(-b)$ using the same digit set $\Sigma_b$
If $x = a_1 a_2 \cdots a_t \in \Sigma_b^*$, then we define
$[x]_b = \sum_{1 \leq i \leq t} a_i (-b)^{t-i}$, and we
also define $(n)_{-b}$ to be the (unique) word representing
$n$ in base $(-b)$, with no leading zeros.
Thus, for example, $[010011]_{-2} = -15$ and $(-15)_{-2} = 10011$.

We use regular expressions to describe the base-$b$ and base-$(-b)$
representations of integers.  For $|b|>10$, however, the individual digits
may consist of numbers that require more than one digit to
write down in base $10$, which creates an ambiguity of how to
interpret, for example, an expression like $12$:  does it represent
a two-digit number or a one-digit number using the single digit $12$?
To avoid this problem, we use brackets to indicate single base-$b$
digits, when necessary.  Thus for a single digit $12$ we write $[12]$.

\section{Automata}

The first step is to show that $f_{3,0} (n)$ is a $(4,3)$-synchronized
function of $n$.   This means there is a DFA
that takes as input, in parallel, the base-$4$
representation of $n$ and the base-$3$ representation of $y$ and
accepts iff $y = f_{3,0} (n)$.

To get the $(4,3)$-synchronized automaton for $f_{3,0} (n)$, we
first ``guess'' it from empirical
data using a version of the Myhill-Nerode theorem 
(see \cite[\S 3.4]{Hopcroft&Ullman:1979} and
\cite[\S 5.7]{Shallit:2022}).
It has 16 states, and is displayed in Figure~\ref{fig1}.
\begin{figure}[htb]
\begin{center}
\includegraphics[width=6.5in]{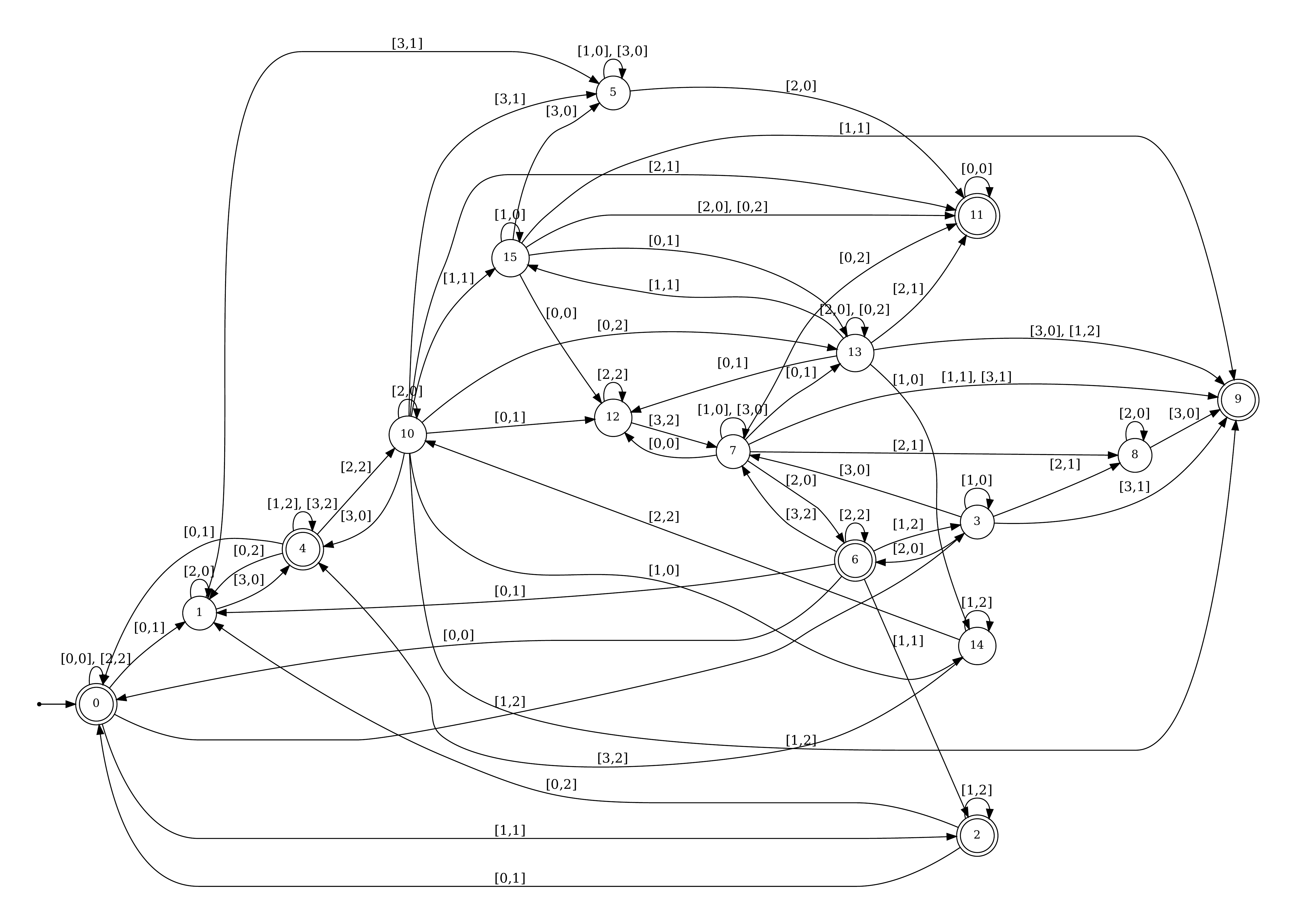}
\caption{$(4,3)$-synchronized automaton for $f_{3,0} (n)$.}
\label{fig1}
\end{center}
\end{figure}

Next, we need to rigorously {\it prove\/} that our guessed automaton is correct.
This could be done tediously with a long induction, but instead
we use {\tt Walnut} \cite{Mousavi:2016,Shallit:2022},
a tool that implements a decision procedure
for the first-order logical theory of the structure $\langle \Enn,
+, V_k \rangle$, where $V_k(n) = k^e$ if $k^e$ is the largest power
of $k$ dividing $n$.  With {\tt Walnut} we can simply state propositions,
phrased in first-order logic, about
automata and automatic sequences, and have {\tt Walnut} verify or
falsify them.

First, let's check that the automaton above
really defines a function from $\Enn$ to $\Enn$; that is,
that for each $n$ there is exactly one $y$ such that the pair
of inputs $(n,y)$ is
accepted, and there are no $n$ for which $(n,y_1)$ and $(n,y_2)$ are
both accepted with $y_1 \not= y_2$:
\begin{verbatim}
eval test30_1 "An Ey $f30(n,y)":
eval test30_2 "~En,y1,y2 $f30(n,y1) & $f30(n,y2) & ?msd_3 y1!=y2":
\end{verbatim}
and {\tt Walnut} returns {\tt TRUE} for both.  

We should briefly explain the syntax of {\tt Walnut}:
\begin{itemize}
\item {\tt A} represents the universal quantifier $\forall$ and
{\tt E} represents the existential quantifier $\exists$
\item {\tt \&} is logical {\tt AND}
\item {\tt |} is logical {\tt OR}
\item {\tt \char'176} is logical {\tt NOT}
\item {\tt =>} is logical implication
\item {\tt msd\_}$b$ indicates that the numbers in the expression
that follow are expressed in base $b$
\item {\tt @}$n$ refers to the value $n$ of an automatic sequence.
\end{itemize}

Next, let's prove by induction on $n$ that the automaton we found
really does compute $f_{3,0} (n)$.   To do so we simply verify
the base case $f_{3,0} (0) = 0$ and the induction step
\begin{equation}
f_{3,0} (n+1) = f_{3,0} (n) + \begin{cases} 
	1, & \text{if $t(3n) = 0$}; \\
	-1, & \text{if $t(3n) = 1$}.
	\end{cases}  \label{induc}.
\end{equation}
We can verify the induction step with the following {\tt Walnut} code:
\begin{verbatim}
def tm4 "0->0110 1->1001":
promote TM4 tm4:
eval test30_3 "?msd_4 An,x (n>=1 & $f30(n, ?msd_3 x) &
   TM4[3*n]=@0) => $f30(n+1, ?msd_3 x+1)":
eval test30_4 "?msd_4 An,x (n>=1 & $f30(n, ?msd_3 x) & 
   TM4[3*n]=@1) => $f30(n+1, ?msd_3 x-1)":
\end{verbatim}
Here the first two commands define the Thue-Morse sequence 
\cite{Allouche&Shallit:1999} in base $4$, and the second two
commands check the correctness of Eq.~\eqref{induc}.
{\tt Walnut} returns {\tt TRUE} for both commands.
This completes the proof that our automaton is correct.

\section{Results for $f_{3,0} (n)$}

Now that we know the automaton really does compute $f_{3,0} (n)$,
we can re-prove the following result of Newman \cite{Newman:1969}:
\begin{theorem}
We have $f_{3,0} (n) > 0$ for all $n \geq 1$.  Furthermore,
$f_{3,0} (n) = O(n^{{\log 4}\over{\log 3}})$ and
this upper bound is optimal.
\end{theorem}

\begin{proof}
For the first claim, let's check that $f_{3,0} (n) > 0$ for all $n \geq 1$:
\begin{verbatim}
eval test30_5 "?msd_4 An,y (n>=1 & $f30(n,y)) => ?msd_3 y>0":
\end{verbatim}

Next, by a theorem about synchronized sequences \cite[Thm.~8]{Shallit:2021h},
we know that if a function 
$g(n)$ is $(a,b)$-synchronized for integers $a, b \geq 2$, then either
\begin{itemize}
\item[(a)] $g(n) = O(n^{{\log b}\over{\log a}})$ and there is a positive
constant $c$ such that $g(n) > c n^{{\log b}\over{\log a}}$ infinitely
often; or
\item[(b)] $g(n) = O(1)$.
\end{itemize}
In our case, $(a,b) = (4,3)$.  Furthermore, we can verify that
$f_{3,0} (n) \not= O(1)$ as follows:
\begin{verbatim}
eval test30_6 "?msd_4 Ay En,x $f30(n,x) & ?msd_3 x>y":
\end{verbatim}
which asserts that for every integer $y$ we can find an $n$
such that $f_{3,0} (n) > y$.  {\tt Walnut} returns {\tt TRUE}
for this command.  Hence alternative (b) is impossible and
alternative (a) holds.  
\end{proof}

The next thing we would like to do is to find, as Newman did,
upper and lower 
bounds on $f_{3,0} (n)/n^{{\log 3}/{\log 4}}$.
However, if we want to compare $f_{3,0} (n)$
to a function like $n^{{\log 3}/{\log 4}} $ using our method, we run into 
the difficulty that no finite automaton can compute arbitrary real powers
like this.  To get around this difficulty, we define a kind of
``pseudopower'' that an automaton can compute.   Using this, we can
prove an improved version of
Newman's original bounds, with very little effort.

Let $a, b$ be integers with $2 \leq a \leq b$.  The particular
pseudopower that we define is
$p_{a,b} (n) := [ (n)_a ]_b$.  In other words, we first express $n$ in base
$a$, and then consider it as a number expressed in
base $b$.  

The following theorem quantifies
how close $p_{a,b} (n)$ is to $n^{{\log b}\over {\log a}}$.
(The same idea was used in \cite{Rampersad&Shallit:2023}, for the
special case where $a = 2$ and $b=4$.)
\begin{theorem}
Suppose $a, b$ are integers with $2 \leq a \leq b$.
Let $e = (\log b)/(\log a)$.
Then 
$$ {\sfrac{a-1}{b-1}} n^e \leq {\sfrac{a-1}{b-1}} ((n+1)^e - 1) \leq [ (n)_a ]_b \leq n^e $$ 
for all $n \geq 0$.
\label{bnd} 
\end{theorem}
\noindent For a proof, see Section~\ref{appendix}.

\bigskip

Now we can prove upper and lower bounds on $f_{3,0} (n)$.   
Let us define $p_{a,b} (n) = [ (n)_a ]_b$ for $n \geq 0$.
Then $p_{3,4}$ is computed by the {\tt Walnut} automaton
defined as follows:
\begin{verbatim}
reg p34 msd_3 msd_4 "([0,0]|[1,1]|[2,2])*":
\end{verbatim}
The idea now is to compute $p_{3,4} (f_{3,0} (n))$, which will
approximate $(f_{3,0} (n))^{{\log 4}\over{\log 3}}$.  The
degree of approximation is given by Theorem~\ref{bnd}.

Next we prove 
\begin{theorem}
For $n \geq 1$ we have $n \leq p_{3,4}(f_{3,0} (n)) \leq (3n-1)/2$,
and both bounds are achieved infinitely often.
\label{b34}
\end{theorem}

\begin{proof}
We use the following {\tt Walnut} code:
\begin{verbatim}
eval bnd1 "?msd_4 An,x,m ($f30(n,x) & $p34(x,m)) => n<=m":
eval bnd2 "?msd_4 An,x,m (n>=1 & $f30(n,x) & $p34(x,m)) => 2*m+1<=3*n":
def bnd3 "?msd_4 Ex $f30(n,x) & $p34(x,n)":
def bnd4 "?msd_4 Ex,m $f30(n,x) & $p34(x,m) & 2*m+1=3*n":
\end{verbatim}
and {\tt Walnut} returns {\tt TRUE} for the first two commands.

The third command creates an automaton accepting the base-$4$
representation of those $n$ for which $p_{3,4}(f_{3,0} (n)) = n$.
Inspection of this automaton reveals that
$n$ is accepted if and only if $(n)_4 \in \{0,2\}^* \{\epsilon, 1\}$.

The fourth command creates an automaton accepting the base-$4$
representation of those $n$ for which $p_{3,4}(f_{3,0} (n)) = (3n-1)/2$.
Inspection of this automaton reveals that
$n$ is accepted if and only if $(n)_4 \in \{1\} \uni 2^*3$.
\end{proof}

We can now prove
\begin{theorem}
Let $e = {{\log 3}\over{\log 4}}$.
Then $c_1 n^e \leq f_{3,0} (n) \leq 
c_2 n^e$,
where $c_1 = 1$ and $c_2 = (9/4)^e \doteq 1.9015$.
\end{theorem}

\begin{proof}
From Theorem~\ref{b34} we have
$n \leq p_{3,4}(f_{3,0} (n))$.  By substituting
$f_{3,0} (n)$ for $n$ in Theorem~\ref{bnd}, we get
$p_{3,4}(f_{3,0} (n)) \leq (f_{3,0}(n))^{1/e}$.
Hence $n \leq (f_{3,0}(n))^{1/e}$, and the 
lower bound follows by raising both sides to the $e$ power.

From Theorem~\ref{b34} we have
$p_{3,4}(f_{3,0} (n)) \leq {3\over 2} n$.
By substituting
$f_{3,0} (n)$ for $n$ in Theorem~\ref{bnd}, we get
$ {2\over 3} (f_{3,0} (n))^{1/e}  \leq p_{3,4}(f_{3,0} (n))$.
Hence ${2\over 3} (f_{3,0} (n))^{1/e} \leq {3\over 2} n$,
so $(f_{3,0} (n))^{1/e} \leq {9 \over 4} n$, and
raising both sides to the $e$ power
gives the bound.
\end{proof}

\begin{remark}
The constants $c_1$ and $c_2$ we obtained can be compared to the
original bounds of $\sfrac{1}{20}$ and $5$ obtained by 
Newman \cite{Newman:1969}, and the
exact upper bound of $(55/3) (3/65)^{{\log 3}\over{\log 4}} 
\doteq 1.602$ obtained by Coquet.
\end{remark}

Furthermore, we can use the synchronized automaton we found
to explore special values of $f_{3,0}$.  For example
\begin{theorem}
\leavevmode
\begin{itemize}
\item[(a)] For $i \geq 0$ we have
$f_{3,0} ( {{85\cdot 4^{i+1} + 1} \over 3} ) = 55 \cdot 3^i$.
\item[(b)] For $i \geq 0$ we have
$f_{3,0} ( 2 \cdot 4^i ) = 2 \cdot 3^i$.
\end{itemize}
\end{theorem}

\begin{proof}
We use the following {\tt Walnut} code:
\begin{verbatim}
reg power43 msd_4 msd_3 "[0,0]*[1,1][0,0]*":
eval bound1 "?msd_4 Ax,y,w,z ($power43(x,y) & 3*w=260*x+1 & 
   ?msd_3 z=55*y) => $f30(w,z)":
eval bound2 "?msd_4 Ax,y,w,z ($power43(x,y) & w=2*x &
   ?msd_3 z=2*y) => $f30(w,z)":
\end{verbatim}
and {\tt Walnut} returns {\tt TRUE} for both.  Note that
{\tt power43}$(x,y)$ asserts that $x = 4^i$ and $y = 3^i$
for some $i \geq 0$.
\end{proof}

\begin{corollary}
Let $e = {{\log 3}\over{\log 4}}$.
\begin{itemize}
\item[(a)] $\limsup_{n \rightarrow \infty} f_{3,0}(n)/n^e 
\geq 55/ (260/3)^e \doteq 1.601958$.
\item[(b)] $\liminf_{n \rightarrow \infty} f_{3,0} (n)/n^e 
\leq 2 / 2^e  \doteq 1.1547$.
\end{itemize}
\end{corollary}

Now that we have obtained results for $f_{3,0} (n)$, we can consider
various analogous sums and generalizations.  This is done
in the next three sections.

\section{Results for $f_{3,1}(n)$ and $f_{3,2} (n)$}

Exactly the same approach can be applied to study
both $f_{3,1}(n)$ and $f_{3,2} (n)$, except
now these functions take non-positive values.
These were studied previously by Dumont \cite{Dumont:1983}
and Drmota and Stoll \cite{Drmota&Stoll:2008}.  

Since {\tt Walnut}'s fundamental domain is $\Enn$,
the natural numbers, it is easiest 
to study the functions $-f_{3,1} (n)$ and $-f_{3,2} (n)$ instead.
Once again we can show that these two functions are
synchronized, with automata of 15 and 14 states, respectively.

\begin{theorem}
We have
\begin{itemize}
\item[(a)] $n/2 \leq p_{3,4}(-f_{3,1}(n)) \leq (3n-1)/2$
for all $\geq 1$.  The lower bound is achieved precisely
when $(n)_4 \in \{0,1\}^*0$, and the upper bound precisely when
$(n)_4 \in \{1\} \uni 2^*3$.
\item[(b)] $(n/2)^{{\log 3}\over{\log 4}} \leq -f_{3,1}(n) \leq
	((9n-3)/4)^{{\log 3}\over{\log 4}}$ for all $n \geq 1$.
\end{itemize}
\end{theorem}

\begin{proof}
Once we have guessed the automaton for $-f_{3,1}(n)$---it has 15 states--we can
verify its correctness as we did for $f_{3,0} (n)$:
\begin{verbatim}
eval test31_1 "An Ey $mf31(n,y)":
eval test31_2 "~En,y1,y2 $mf31(n,y1) & $mf31(n,y2) & ?msd_3 y1!=y2":
eval test31_3 "?msd_4 An,x (n>=1 & $mf31(n, ?msd_3 x) &
   TM4[3*n+1]=@1) => $mf31(n+1, ?msd_3 x+1)":
eval test31_4 "?msd_4 An,x (n>=1 & $mf31(n, ?msd_3 x) &
   TM4[3*n+1]=@0) => $mf31(n+1, ?msd_3 x-1)":
eval test31_5 "?msd_4 An,y (n>=1 & $mf31(n,y)) => ?msd_3 y>0":
eval test31_6 "?msd_4 Ay En,x $mf31(n,x) & ?msd_3 x>y":
\end{verbatim}
Then we use {\tt Walnut} to verify the bounds in the theorem.
\begin{verbatim}
eval test31_bnd1 "?msd_4 An,x,m ($mf31(n,x) & $p34(x,m)) => n<=2*m":
eval test31_bnd2 "?msd_4 An,x,m (n>=1 & $mf31(n,x) & $p34(x,m)) => 2*m+1<=3*n":
def test31_bnd3 "?msd_4 Ex $mf31(n,x) & $p34(x,n)":
def test31_bnd4 "?msd_4 Ex,m $mf31(n,x) & $p34(x,m) & 2*m+1=3*n":
\end{verbatim}
The last two commands produce automata accepting the sets of the theorem.
\end{proof}

Dumont \cite{Dumont:1983} also
studied $f_{3,2} (n)$, but erroneously claimed that 
it takes arbitrarily large values that are both positive and
negative. This
was corrected by Drmota and Stoll \cite{Drmota&Stoll:2008}:  $f_{3,2} (n)$
is always nonpositive.  Using our method we can prove
the following result:

\begin{theorem}
We have
\begin{itemize}
\item[(a)] $0 \leq p_{3,4} (-f_{3,2} (n)) \leq (3n+1)/4$ for $n \geq 0$.
The lower bound is achieved  precisely when 
$(n)_4 \in \{0,2\}^*$, and the upper bound 
precisely when $n =  4^i$, $i \geq 0$.

\item[(b)] $0 \leq -f_{3,2}(n) \leq ((9n+3)/8)^{{\log 3}\over {\log 4}} $
for all $n \geq 0$.
\end{itemize}
\end{theorem}

\section{Results for $f_{5,j} (n)$}
\label{sec5}

Similarly, we can study $f_{5,j} (n)$.  For example, for $f_{5,0} (n)$ we have
the following result.
\begin{theorem}
We have
\begin{itemize}
\item[(a)] $(47n+140)/176 \leq p_{5,16} (f_{5,0} (n)) \leq (15n-11)/4$
for all $n \geq 2$.  The lower bound is achieved precisely when
$(n)_{16} \in [11]^*[12]$, and the upper bound is achieved
precisely when $(n)_{16} \in \{1,5\} \uni 44^*5$.

\item[(b)] $((47n+140)/176)^{{\log 5}\over{\log 16}} \leq f_{5,0} (n)
\leq ((225n-165)/16)^{{\log 5}\over {\log 16}} $ for all
$n \geq 2$.
\end{itemize}
\end{theorem}

This follows by guessing an automaton
for $f_{5,0} (n)$---it has $26$ states---and
mimicking the proofs of previous sections.
We leave the proof via {\tt Walnut} to the reader.

For $j = 1,2,3,4$, however, $f_{5,j} (n)$ takes on arbitrarily large
positive and negative values.  To understand these values easily, we need
to be able to work with negative numbers in {\tt Walnut}, and this
is done using negative bases \cite{Shallit&Shan&Yang:2022}.
Let us illustrate this with the case of $j = 1$.  The remaining
cases are left to the reader to explore.

We can show that $f_{5,j} (n)$ is $(16,-5)$-synchronized, meaning that
there is an automaton accepting $(n,y)$ where $y= f_{5,j} (n)$
and $n$ is expressed in base $16$ and $y$ is expressed in base $-5$.
It has $68$ states.
\begin{theorem}
We have
\begin{itemize}
\item[(a)] $p_{5,16}(-f_{5,1}(n)) \leq (5n-3)/2$, with equality
iff $(n)_{16} \in 6^* 7$.
\item[(b)] $p_{5,16} (f_{5,1}(n)) \leq (121n-463)/3412$, with
equality iff $(n)_{16} \in 1[12]7 \uni 1[12]6[14]^*[15]$.
\item[(c)] $p_{5,16} (f_{5,1} (n)) = 0$ iff
$(n)_{16} \in (\{0,2\} \uni 1\{7,9,[11],[13],[15]\}^*\{8,[10],[12],[14]\})^*$.
\end{itemize}
\end{theorem}

\begin{proof}
The following commands check that our guessed automaton is correct, and
that $f_{5,1} (n)$ takes positive and negative values of unbounded 
magnitude.
\begin{verbatim}
morphism tm16 "0->0110100110010110 1->1001011001101001":
promote TM16 tm16:
eval test51_1 "An Ey $f51(n,y)":
eval test51_2 "~En,y1,y2 $f51(n,y1) & $f51(n,y2) & ?msd_neg_5 y1!=y2":
eval test51_3 "?msd_16 An,x (n>=1 & $f51(n, ?msd_neg_5 x) &
   TM16[5*n+1]=@0) => $f51(n+1, ?msd_neg_5 x+1)":
eval test51_4 "?msd_16 An,x (n>=1 & $f51(n, ?msd_neg_5 x) &
   TM16[5*n+1]=@1) => $f51(n+1, ?msd_neg_5 x-1)":
eval test51_5 "?msd_16 Ay En,x $f51(n,x) & ?msd_neg_5 x>y":
eval test51_6 "?msd_16 Ay En,x $f51(n,x) & ?msd_neg_5 x<y":
\end{verbatim}

Next, we need an automaton that converts values from base $-5$ to base $5$,
or, more precisely, a $(-5,5)$-synchronized automaton that computes
the function $\max(0,n)$.  
This is easy to construct, and is displayed in Figure~\ref{conv5}.
\begin{figure}[htb]
\begin{center}
\includegraphics[width=6.5in]{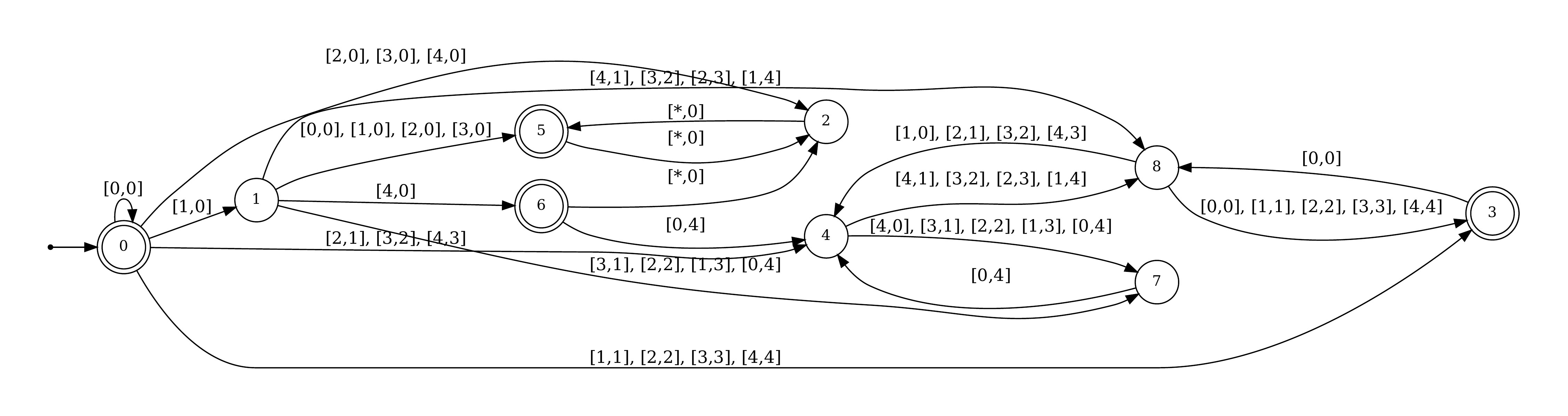}
\caption{$(-5,5)$-synchronized automaton for $\max(0,n)$.}
\label{conv5}
\end{center}
\end{figure}
We omit the proof of correctness.

We now verify claim (a):
\begin{verbatim}
reg p165 msd_16 msd_5 "([0,0]|[1,1]|[2,2]|[3,3]|[4,4])*":
eval negvalues51 "?msd_16 An,x,y,w ((?msd_neg_5 x<0) & 
   $f51(n,?msd_neg_5 x) & $conv55((?msd_neg_5 _x),?msd_5 y) &
   $p165(w,?msd_5 y)) => 2*w+3<=5*n":
eval negvalues51_match "?msd_16 Ex,y,w (?msd_neg_5 x<0) &
   $f51(n,?msd_neg_5 x) & $conv55((?msd_neg_5 _x),?msd_5 y) &
   $p165(w,?msd_5 y) & 2*w+3=5*n":
\end{verbatim}
and {\tt Walnut} returns {\tt TRUE} for the first command.
For the second, {\tt Walnut} computes an automaton accepting
the values specified.

In principle, the second claim could be verified in exactly
the same way.   However, here we run across a limitation
in the current version of {\tt Walnut}, so
have to special-case the multiplication by $3412$.
\begin{verbatim}
def mult2 "?msd_16 x=2*y":
def mult3 "?msd_16 x=3*y":
def mult4 "?msd_16 Ez $mult2(x,z) & $mult2(z,y)":
def mult12 "?msd_16 Ez $mult3(x,z) & $mult4(z,y)":
def mult13 "?msd_16 Ez x=y+z & $mult12(z,y)":
def mult52 "?msd_16 Ez $mult4(x,z) & $mult13(z,y)":
def mult53 "?msd_16 Ez x=y+z & $mult52(z,y)":
def mult212 "?msd_16 Ez $mult4(x,z) & $mult53(z,y)":
def mult213 "?msd_16 Ez x=y+z & $mult212(z,y)":
def mult852 "?msd_16 Ez $mult4(x,z) & $mult213(z,y)":
def mult853 "?msd_16 Ez x=y+z & $mult852(z,y)":
def mult3412 "?msd_16 Ez $mult4(x,z) & $mult853(z,y)":
def pv51 "?msd_16 Ex,y (?msd_neg_5 x>=0) & $f51(n,?msd_neg_5 x)
   & $conv55((?msd_neg_5 x),?msd_5 y) & $p165(w,?msd_5 y)":
eval f51pcheck "?msd_16 An,w,t (n>=30 & $pv51(n,w) & 
   $mult3412(t,w)) => t+463<=121*n":
def f51p_equal "?msd_16 Ew,t n>=30 & $pv51(n,w) & $mult3412(t,w)
   & t+463=121*n":
\end{verbatim} 
{\tt Walnut} returns {\tt TRUE} for the claim {\tt f51pcheck}.
The command {\tt f51p\_equal} returns a 5-state automaton
that accepts the language specified.

Finally, to verify (c) we write
\begin{verbatim}
def f51eq0 "?msd_16 $f51(n,?msd_neg_5 0)":
\end{verbatim}
and the result is a $2$-state automaton accepting the language specified.
\end{proof}

Our approach can also be used to study the rarefied sums $f_{7,j} (n)$ for
$0 \leq j \leq 6$.   This raises the natural question of whether
$f_{p,j} (n)$ might be $(2^{p-1},p)$-synchronized in general.
However, as shown by \cite{Grabner&Herendi&Tichy:1997}, this is
not the case for $p = 17$.  However, it is possible that
$f_{p,j} (n)$ might be synchronized in other, more exotic, number
systems.

\section{Rarefied sums for an analogue of Thue-Morse}

Our method can also be used to explore various analogous sums.
In this section we give an example.

Let $r(n)$ be the analogue of $t(n)$, except now we count the
parity of the number of $0$'s in the base-$2$ representation of $n$, instead
of the number of $1$'s.   Note that $r(0) = 0$.  
In analogy with $f_{b,j} (n)$ define
$g_{b,j}(n) = \sum_{0 \leq i < n} (-1)^{r(bi+j)} $.

Once again the function $n \rightarrow g_{3,0}(n)$ is
$(4,3)$-synchronized, this time with an automaton with $18$ states.
With it we can easily prove
\begin{theorem}
We have
\begin{itemize}
\item[(a)] $g_{3,0} (n) > 0$ for all $n \geq 1$.
\item[(b)] $p_{3,4}(g_{3,0}(n)) \leq (3n+2)/4$ 
for all $n \geq 0$, with equality exactly when
$n = (4^i+2)/3$ for $i \geq 1$.
Hence $\limsup_{n \rightarrow\infty} p_{3,4}(g_{3,0}(n))/n = {3/4}$.
\item[(c)] $g_{3,0} (n) \leq ((9n+6)/8)^{{\log 3}\over{\log 4}}$
for all $n\geq 0$.
\item[(d)] $g_{3,0}(n) = 1$ exactly when
$n = 2\cdot 4^i + 1$ for $i \geq 0$. 
\end{itemize}
\end{theorem}

We omit the details.

\section{Proof of Theorem~\ref{bnd}}
\label{appendix}

In this section, we prove the needed technical result, Theorem~\ref{bnd}.
We start with some useful lemmas.

\begin{lemma}
Suppose $n \geq 1$, $0 \leq x \leq 1$, and $e \geq 1$.
Then $(n+x)^e \leq (n+1)^e + e(x-1)$.
\label{lem1}
\end{lemma}

\begin{proof}
For $x \in [0,1]$, define $h(x) := (n+x)^e - (n+1)^e - e(x-1)$.
We want to prove that $h(x) \leq 0$.
Now $h'(x) = e(n+x)^{e-1} - e$. From the hypotheses on $n,x,$ and $e$, we 
see that $h'(x) \geq 0$. Hence $h$ is increasing on 
the interval $[0,1]$. Thus $h(x) \leq h(1) = 0$. 
\end{proof}

\begin{lemma}
Suppose $1 \leq a \leq b$.
Then ${{a-1}\over a}\log b \geq {{b-1}\over b} \log a$.
\label{lem2}
\end{lemma}

\begin{proof}
Writing $b = a+x$ with $x \geq 0$, we must show
$h(x) = {{a-1}\over a} \log(a+x) - {{a+x-1}\over{a+x}}\log a \geq 0$.
But $h'(x) = ((a-1)(a+x) - a\log a)/(a (a+x)^2)$.
Now from the Taylor series expansion we have
$e^{a-1} \geq a$, so $a-1 \geq \log a$, so
$a(a-1) \geq a \log a$, so $h'(x) \geq 0$.
Now $h(0) = 0$, so $h(x) \geq 0$ for all $x$.
\end{proof}

\begin{lemma}
Suppose $1 < a\leq b $ and $x \in [0,1]$.   Then
$$ ((a-1)x+1)^{\log b} \leq ((b-1)x+1)^{\log a}.$$
\label{lem3}
\end{lemma}

\begin{proof}
By taking logarithms, it suffices to show that
$$(\log b)(\log ((a-1)x+1)) \leq (\log a)(\log ((b-1)x + 1)) .$$
This is true for $x = 0$, so it suffices to prove
$ {{\log b}\over{\log a}} \leq f(x)$
for $f(x) = {{\log ((b-1)x+1)} \over {\log((a-1)x+1)}} $
and $x \in (0,1]$.
If we show that $f(x)$ is decreasing on $x\in (0,1]$ then
$f(x) \geq f(1) = {{\log b}\over{\log a}}$, as desired.

To show $f(x)$ is decreasing on $x \in (0,1]$, we should
show $f'(x) \leq 0$. 
We have
$$ f'(x) = {{g(x)} \over {((a-1)x + 1)((b-1)x+1) (\log((a-1)x + 1))^2}}  .$$
where 
$$g(x) = (b-1)((a-1)x+1) \log( (a-1)x + 1) -
(a-1)((b-1)x+1) \log ((b-1)x + 1) .$$ 
Since the denominator of $f'(x)$ is positive, it suffices to show
that $g(x) \leq 0$.  To show this, since $g(0) = 0$,
it suffices to show $g(x)$ is decreasing on $[0,1]$.
Now 
$$g'(x) = (a-1)(b-1) (\log(1+(a-1)x) - \log(1+(b-1)x)) $$
which, since $a\leq b$, gives $g'(x) \leq 0$.
The result now follows.
\end{proof}

We are now ready to prove Theorem~\ref{bnd}.

\begin{proof}[Proof of Theorem~\ref{bnd}]
The first inequality 
${{a-1}\over{b-1}} n^e \leq {{a-1}\over{b-1}} ((n+1)^e - 1)$
follows from Lemma~\ref{lem1} with $x = 0$.    So we focus
on the remaining inequalities.

The upper and lower bounds clearly hold for $n = 0$.

Now suppose $n \geq 1$.  For the upper bound, express $n$
in base $a$ as follows:  $n = c_{t-1} a^{t-1} + \cdots + c_1 a + c_0$.
Then $[(n)_a]_b = c_{t-1} b^{t-1} + \cdots + c_1 b + c_0$.
Then 
\begin{align*}
n^e &= (c_{t-1} a^{t-1} + \cdots + c_1 a + c_0)^e  \\
&\geq c_{t-1} (a^e)^{t-1} + \cdots + c_1 (a^e) + c_0  \\
&= c_{t-1} b^{t-1} + \cdots + c_1 b + c_0 \\
&= [(n)_a]_b.  
\end{align*}

Next, let us prove the lower bound by induction on $n$.
The base case is $1 \leq n \leq a-1$.  For these $n$ we need to show
$$ \sfrac{a-1}{b-1} ((n+1)^e -1) \leq n,$$ or, equivalently by
setting $x = n/(a-1)$ and rearranging, we need to show
$(1+(a-1)x)^{\log b} \leq (1+(b-1)x)^{\log a}$ for $x \in [0,1]$.  But
this is Lemma~\ref{lem3}.

Now assume the lower bound holds for $n'< n$; we prove it for $n$.
Write $n = ak+c$ for some integer $c$, $0 \leq c \leq a-1$.
Then
\begin{align}
& \sfrac{a-1}{b-1} ((n+1)^e - 1)  \\
&\quad = \sfrac{a-1}{b-1} ((ak+c+1)^e - 1)  \nonumber\\
&\quad= \sfrac{a-1}{b-1} \bigl( a^e \bigl(k+ \sfrac{c+1}{a} \bigr)^e - 1\bigr)  \nonumber\\
&\quad= \sfrac{a-1}{b-1} \bigl(b \bigl(k + \sfrac{c+1}{a} \bigr)^e -1 \bigr)  \nonumber\\ 
&\quad\leq  \sfrac{a-1}{b-1}  \bigl( b \bigl( (k+1)^e + e \sfrac{c+1-a}{a} \bigr)   - 1 \bigr) \quad
	\text{(by Lemma~\ref{lem1} with $n=k$, $x = (c+1)/a$))} \nonumber \\
&\quad\leq b \bigl( \sfrac{a-1}{b-1} ((k+1)^e - 1) \bigr) + c \label{gap} \\
&\quad\leq b [(k)_a]_b + c \quad \text{(by induction)} \nonumber\\ 
&\quad= [(ak+c)_a]_b \nonumber\\
&\quad= [ (n)_a]_b. \nonumber
\end{align}
However, the inequality \eqref{gap}
$$ \sfrac{a-1}{b-1}  \bigl( b \bigl((k+1)^e + e \sfrac{c+1-a}{a} \bigr)   - 1 \bigr)
\leq  b \bigl( \sfrac{a-1}{b-1} ((k+1)^e - 1) \bigr) + c$$
still needs justification.  To see it, note that by simplification
it is equivalent to the claim 
that $(abe-ab-be+a)(a-1-c) \geq 0$.  The second term is non-negative,
since $c\leq a-1$, so it suffices to show
$abe-ab -be +a \geq 0$, 
which is equivalent to the claim
$ {{a-1}\over a} \log b \geq {{b-1} \over b} \log a$.  But this
is Lemma~\ref{lem2}.
\end{proof}

\begin{remark}
We remark that the bounds in Theorem~\ref{bnd} are tight.
The lower bound, for example, is approached by
numbers of the form $a^n - 1$, while the upper bound is
reached for numbers of the form $a^n$.
\end{remark}

\section*{Acknowledgments}

I thank Jean-Paul Allouche very warmly for his help in proving
Theorem~\ref{bnd}, and for sharing his deep knowledge of
the literature on Thue-Morse sums.  I also thank
Daniela Opo\v{c}ensk\'a for proofreading.


\begin{thebibliography}{99}

\bibitem{Allouche&Shallit:1999}
J.-P. Allouche and J.~O. Shallit.
\newblock The ubiquitous {Prouhet-Thue-Morse} sequence.
\newblock In C.~Ding, T.~Helleseth, and H.~Niederreiter, editors, {\em
  Sequences and Their Applications, Proceedings of SETA '98}, pp.  1--16.
  Springer-Verlag, 1999.

\bibitem{Boreico&El-Baz&Stoll:2014}
I. Boreico, D. El-Baz, and T. Stoll.
\newblock On a conjecture of Dekking: the sum of digits of even numbers.
\newblock {\it J. Th\'eor. Nombres Bordeaux} {\bf 26} (2014), 17--24.

\bibitem{Coquet:1983}
J.~Coquet.
\newblock A summation formula related to the binary digits.
\newblock {\em Inventiones Math.} {\bf 73} (1983), 107--115.

\bibitem{Dartyge&Tenenbaum:2005}
C. Dartyge and G. Tenenbaum.
\newblock Sommes des chiffres de multiples d'entiers.
\newblock {\it Ann. Inst. Fourier} {\bf 55} (2005), 2423--2474.

\bibitem{Drmota&Skalba:1995}
M. Drmota and M. Ska{\l}ba.
\newblock Sign-changes of the Thue--Morse fractal function and
Dirichlet L-series.
\newblock {\it Manuscripta Math.} {\bf 86} (1995), 519--541.

\bibitem{Drmota&Skalba:1999}
M. Drmota and M. Ska{\l}ba.
\newblock Rarified sums of the Thue--Morse sequence.
\newblock {\it Trans. Amer. Math. Soc.} {\bf 352} (1999), 609--642.

\bibitem{Drmota&Stoll:2008}
M. Drmota and T. Stoll.
\newblock Newman's phenomenon for generalized Thue--Morse sequences.
\newblock {\it Discrete Math.} {\bf 308} (2008), 1191--1208.

\bibitem{Dumont:1983}
J.-M. Dumont.
\newblock Discr{\'e}pance des progressions arithm{\'e}tiques
dans la suite de Morse.
\newblock {\it C. R. Acad. Sci. Paris, S\'erie I} {\bf 297} (1983), 145--148.

\bibitem{Goldstein&Kelly&Speer:1992}
S.~Goldstein, K.~Kelly, and E.~R. Speer.
\newblock The fractal structure of rarefied sums of the {Thue-Morse} sequence.
\newblock {\em J. Number Theory} {\bf 42} (1992), 1--19.

\bibitem{Grabner:1993}
P. J. Grabner.
\newblock A note on the parity of the sum-of-digits function.
\newblock {\it S\'em. Lotharingien de Combinatoire}
{\bf 30} (1993), Paper B30e.  Available at
\url{https://eudml.org/doc/124434}.

\bibitem{Grabner&Herendi&Tichy:1997}
P. J. Grabner, T. Herendi, and R. F. Tichy.
\newblock Fractal digital sums and codes.
\newblock {\it Appl. Algebra Engng. Comm. Comput.}
{\bf 8} (1997), 33--39.

\bibitem{Hofer:2011}
R. Hofer.
\newblock Coquet-type formulas for the rarefied weighted Thue--Morse sequence.
\newblock {\it Discrete Math.} {\bf 311} (2011), 1724--1734.

\bibitem{Hopcroft&Ullman:1979}
J.~E. Hopcroft and J.~D. Ullman.
\newblock {\em Introduction to Automata Theory, Languages, and Computation}.
\newblock Addison-Wesley, 1979.

\bibitem{Larcher&Zellinger:2011}
G. Larcher and H. Zellinger.
\newblock On irregularities of distribution of weighted sums-of-digits.
\newblock {\it Discrete Math.} {\bf 311} (2011), 109--123.

\bibitem{Mousavi:2016}
H.~Mousavi.
\newblock Automatic theorem proving in {{\tt Walnut}}.
\newblock Arxiv preprint arXiv:1603.06017 [cs.FL], available at
  \url{http://arxiv.org/abs/1603.06017}, 2016.

\bibitem{Newman:1969}
D.~J. Newman.
\newblock On the number of binary digits in a multiple of three.
\newblock {\em Proc. Amer. Math. Soc.} {\bf 21} (1969), 719--721.

\bibitem{Newman&Slater:1975}
D. J. Newman and M. Slater.
\newblock Binary digit distribution over naturally defined sequences.
\newblock {\it Trans. Amer. Math. Soc.} {\bf 213} (1975), 71--78.

\bibitem{Rampersad&Shallit:2023}
N. Rampersad and J. Shallit.
\newblock Rudin-Shapiro sums via automata theory and logic.
\newblock ArXiv preprint arXiv:2302.00405 [math.NT], February 2023.
\newblock Available at \url{https://arxiv.org/abs/2302.00405}.

\bibitem{Shallit:2021h}
J.~Shallit.
\newblock Synchronized sequences.
\newblock In T.~Lecroq and S.~Puzynina, editors, {\em WORDS 2021}, Vol. 12847
  of {\em Lecture Notes in Computer Science}, pp.  1--19. Springer-Verlag,
  2021.

\bibitem{Shallit:2022}
J.~Shallit.
\newblock {\em The Logical Approach To Automatic Sequences: Exploring
  Combinatorics on Words with {\tt Walnut}}, Vol. 482 of {\em London Math.
  Society Lecture Note Series}.
\newblock Cambridge University Press, 2022.

\bibitem{Shallit&Shan&Yang:2022}
J. Shallit, S. L. Shan, and K. H. Yang.
\newblock Automatic sequences in negative bases and proofs of some conjectures of Shevelev.
\newblock ArXiv preprint arXiv:2208.06025 [cs.FL], August 11 2022.
\newblock Available at \url{https://arxiv.org/abs/2208.06025}.

\bibitem{Shevelev:2008}
V. Shevelev.
\newblock Generalized Newman phenomena and digit conjectures on primes.
\newblock {\it Int.  J. Math. Math. Sci.} (2008), ID 908045.

\bibitem{Shevelev:2009}
V. Shevelev.
\newblock Exact exponent in the remainder term
of Gelfond's digit theorem in the binary case.
\newblock {\it Acta Arith.} {\bf 136} (2009), 91--100.

\bibitem{Stoll:2005}
T. Stoll.
\newblock Multi-parametric extensions of Newman's phenomenon.
\newblock {\it Integers} {\bf 5} (3) (2005), Paper \#14.

\bibitem{Tenenbaum:1997}
G. Tenenbaum.
\newblock Sur la non-d\'erivabilit\'e de fonctions p\'eriodiques
associ\'ees \`a certaines formules sommatoires.
\newblock In R. L. Graham and J. Ne\v{s}et\v{r}il, eds.,
{\it The Mathematics of Paul Erd\H{o}s}, Algorithms and Combinatorics,
Vol.~13, Springer-Verlag, 1997, pp.~117--128.


\end{thebibliography}
\end{document}